\documentclass[final]{article}

\usepackage[a4paper,left=42mm,right=42mm]{geometry}

\usepackage[utf8]{inputenc}

\usepackage{amssymb,amsmath,amsthm}
\theoremstyle{plain} 
    \newtheorem{theorem}{Theorem}[section]
    \newtheorem*{theorem*}{Theorem}
    \newtheorem{corollary}[theorem]{Corollary}
    \newtheorem{lemma}[theorem]{Lemma}
    \newtheorem{proposition}[theorem]{Proposition}
\theoremstyle{definition}
    \newtheorem{definition}[theorem]{Definition}
    
\theoremstyle{remark}    
    \newtheorem{remark}[theorem]{Remark}
    \newtheorem{claim}{Claim}[theorem]

\usepackage{bm} 

\usepackage[shortcuts]{extdash}

\usepackage{IEEEtrantools}

\newcommand{\range}{\operatorname{range}}

\newcommand{\satisfies}{\models}

\newcommand{\crit}{\operatorname{crit}}
\newcommand{\rank}{\operatorname{rank}}

\newcommand{\ZFC}{\mathrm{ZFC}}
\newcommand{\ZF}{\mathrm{ZF}}
\newcommand{\rBC}{\mathrm{rBC}}
\newcommand{\BC}{\mathrm{BC}}

\newcommand{\OR}{\mathrm{OR}}

\newcommand{\ot}{\operatorname{ot}}
\newcommand{\dbar}[1]{\bar{\bar #1}}
\newcommand{\cof}{\operatorname{cof}}

\newcommand{\VP}{\mathrm{VP}}
\newcommand{\boldVP}{\mathrm{\mathbb{VP}}}
\newcommand{\con}[1]{\operatorname{Con}(#1)}
\newcommand{\AC}{\mathrm{AC}}

\usepackage{comment}

\usepackage{soul}

\usepackage{enumitem}
\setlist[enumerate]{label={\upshape(\roman*)},noitemsep}

\usepackage{tabularx}
\newcolumntype{C}{>{\centering\arraybackslash$}X<{$}}
\newcolumntype{M}{>{\centering\arraybackslash$}p{.35cm}<{$}}

\usepackage{todonotes}

\usepackage{biblatex}
\renewbibmacro{in:}{}
\makeatletter
\DeclareFontFamily{OMX}{MnSymbolE}{}
\DeclareSymbolFont{MnLargeSymbols}{OMX}{MnSymbolE}{m}{n}
\SetSymbolFont{MnLargeSymbols}{bold}{OMX}{MnSymbolE}{b}{n}
\DeclareFontShape{OMX}{MnSymbolE}{m}{n}{
    <-6>  MnSymbolE5
   <6-7>  MnSymbolE6
   <7-8>  MnSymbolE7
   <8-9>  MnSymbolE8
   <9-10> MnSymbolE9
  <10-12> MnSymbolE10
  <12->   MnSymbolE12
}{}
\DeclareFontShape{OMX}{MnSymbolE}{b}{n}{
    <-6>  MnSymbolE-Bold5
   <6-7>  MnSymbolE-Bold6
   <7-8>  MnSymbolE-Bold7
   <8-9>  MnSymbolE-Bold8
   <9-10> MnSymbolE-Bold9
  <10-12> MnSymbolE-Bold10
  <12->   MnSymbolE-Bold12
}{}
\DeclareMathDelimiter{\ulcorner}
    {\mathopen}{MnLargeSymbols}{'036}{MnLargeSymbols}{'036}
\DeclareMathDelimiter{\urcorner}
    {\mathclose}{MnLargeSymbols}{'043}{MnLargeSymbols}{'043}
\makeatother

\title{Berkeley Cardinals and Vop\v enka's Principle}
\author{Marwan Salam Mohammd
\\
    \texttt{marwan.mizuri@gmail.com}
    }

\begin{document}

\maketitle

\begin{abstract}
    We introduce ``$n$-choiceless'' supercompact and extendible cardinals in Zermelo-Fraenkel set theory without the Axiom of Choice. We prove relations between these cardinals and Vop\v enka's Principle similar to those of Bagaria's work in ``$C^{(n)}$-Cardinals'' and ``More on the Preservation of Large Cardinals Under Class Forcing.'' We use these relations to characterize Berkeley cardinals in terms of a restricted form of Vop\v enka's Principle. Finally, we determine the consistency strength of some relevant theories that arise.
\end{abstract}


\section{Introduction}\label{section_intro}

In \citeyear{reinhardt}, William N. Reinhardt introduced in his PhD dissertation the large cardinal notion now called a Reinhardt cardinal \cite{reinhardt}. A cardinal is called a \emph{Reinhardt cardinal} iff it is the critical point of an elementary embedding $j:V\rightarrow V.$ Soon after, Kenneth Kunen \cite{kunen} found an inconsistency of that notion with $\ZFC,$ using the Axiom of Choice ($\AC$) in his proof. The question remains open to this day whether $\AC$ is absolutely necessary for this refutation.

\begin{theorem}[\citeauthor{kunen} \cite{kunen}]
There is no nontrivial elementary embedding $j:V\rightarrow V.$
\end{theorem}

In his Berkeley set theory graduate course around 1990, W. Hugh Woodin introduced the concept of a Berkeley cardinal (Definition~\ref{def:berkeley}), a notion stronger than a Reinhardt cardinal, as an exercise for his students to explore potential inconsistencies. But, despite the passage of almost half a century, no inconsistencies have been found. Reinhardt cardinals, Berkeley cardinals, and a few other variations on these two large cardinal notions are the topics of the article ``Large Cardinals Beyond Choice'' \cite{lcbc} by Joan Bagaria, Peter Koellner, and Woodin.

In this paper, we relate Berkeley cardinals to a very well known large cardinal notion called Vop\v enka's Principle. Vop\v enka's Principle, $\boldVP,$ states that for any proper class of structures of the same type, there exist two distinct members in the class such that one is elementarily embeddable into the other. The precise formulation of this notion will be given in Section~\ref{section_vopenka_and_choiceless_extendibles}.

Since we are working with Berkeley cardinals, our background theory will be $\ZF$ without $\AC,$ unless otherwise stated. Let $\boldVP^{\omega}$ denote $\boldVP$ restricted to proper classes of structures of the same finite type. The main results of this paper are facts in $\ZF + \neg\boldVP:$

\theoremstyle{plain}
\newtheorem*{thm:vp_to_berkeley_levels}{Theorem~\ref{vp_to_berkeley_levels}}
\begin{thm:vp_to_berkeley_levels}[$\ZF + \neg\boldVP$]
    If $\boldVP^\omega$ holds, then there is an ordinal $\delta>\omega$ for which $\boldVP^\delta$ holds while $\boldVP^{\delta+1}$ fails, and moreover, $\delta$ is a Berkeley cardinal.
\end{thm:vp_to_berkeley_levels}

\newtheorem*{cor:vp_to_berkeley_levels_sup}{Corollary~\ref{cor:vp_to_berkeley_levels_sup}}
\begin{cor:vp_to_berkeley_levels_sup}[$\ZF + \neg\boldVP$]
    If $\boldVP^\omega$ holds, then
    $$
    \sup\{\mu\mid \boldVP^\mu\textrm{ holds}\}=\sup\{\delta\mid \delta \text{ is Berkeley}\}.
    $$
\end{cor:vp_to_berkeley_levels_sup}

These results raise questions about the consistency strength of $\ZF+ \neg\boldVP+\boldVP^{\omega}.$ In regards to this, we prove the following:

\newtheorem*{thm:equiconsistency_theories_berkeley_and_finite_vp}{Theorem~\ref{equiconsistency_theories_berkeley_and_finite_vp}}
\begin{thm:equiconsistency_theories_berkeley_and_finite_vp}
    The theory $\ZF+\BC$ is equiconsistent with $\ZF+\boldVP^\omega+\neg\boldVP.$
\end{thm:equiconsistency_theories_berkeley_and_finite_vp}

Rank-Berkeley cardinals (Definition~\ref{def:rank_berkeley}) are a natural weakening of Berkeley cardinals first discovered by Farmer Schlutzenberg and Woodin, independently, when they realized that their existence follows from the existence of a Reinhardt cardinal. We will establish analogues of theorems \ref{vp_to_berkeley_levels}, \ref{cor:vp_to_berkeley_levels_sup}, and \ref{equiconsistency_theories_berkeley_and_finite_vp} for rank-Berkeley cardinals as well. As an application of that, we get the corollary below.
\newtheorem*{cor:weakest_vp_to_strongest_vp}{Corollary~\ref{weakest_vp_to_strongest_vp}}
\begin{cor:weakest_vp_to_strongest_vp}[$\ZFC$]
    $\boldVP$ restricted to definable, without parameters, classes of structures of the same finite type implies (and hence is equivalent to) $\boldVP.$
\end{cor:weakest_vp_to_strongest_vp}
\noindent That is, under $\AC,$ the weakest possible form of $\boldVP$ is equivalent to its strongest form. We are not aware if there is a more direct proof of this result.

In \citetitle{cn-cardinals} \cite{cn-cardinals}, Bagaria establishes an exact relation between $\boldVP$ and what he calls $C^{(n)}$\-/extendible cardinals (see Definition~\ref{cn_extendibles} for an equivalent definition).
\begin{theorem}[\citeauthor{cn-cardinals} \cite{cn-cardinals}]\label{bagaria_cn_extendibles_and_vp}
    Assuming $\AC,$ for $n\geq 1,$ the following are equivalent:
    \begin{enumerate}
        \item $\boldVP$ restricted to $\Pi_{n+1}$-definable, with parameters, classes of structures.
        \item There is a proper class of $C^{(n)}$\-/extendible cardinals.
    \end{enumerate}
\end{theorem}
\noindent The proof uses the following alternative form of Kunen's inconsistency:
\begin{theorem}[{\cite[Corollary 23.14(a)]{kanamori}}]\label{kunen_inconsistency_alternare_form}
    Assuming $\AC,$ for any $\delta,$ there is no nontrivial elementary embedding $j:V_{\delta+2}\rightarrow V_{\delta+2}.$
\end{theorem}
\noindent
Theorem~\ref{vp_to_berkeley_levels} comes from the struggle of bringing Theorem~\ref{bagaria_cn_extendibles_and_vp} into the Choiceless context. Thus, we introduce the $n$-choiceless extendible cardinals to play the role of $C^{(n)}$\-/extendible cardinals in Bagaria's work, without relying on $\AC.$ We establish a characterization parallel to Theorem~\ref{bagaria_cn_extendibles_and_vp}, namely:

\newtheorem*{thm:equivalence_of_vp_and_n_choiceless_extendibles}{Theorem~\ref{equivalence_of_vp_and_n_choiceless_extendibles}}
\begin{thm:equivalence_of_vp_and_n_choiceless_extendibles}
    For $n\geq1,$ the following are equivalent:
    \begin{enumerate}
        \item $\boldVP$ restricted to $\Pi_{n+1}$-definable, with ordinal parameters, classes of structures.
        \item There is a proper class of $n$-choiceless extendible cardinals.
        \item $\boldVP$ restricted to $\Pi_{n+1}$-definable, with parameters, classes of structures.
    \end{enumerate}
\end{thm:equivalence_of_vp_and_n_choiceless_extendibles}

We also consider the consistency strength of the failure of Theorem~\ref{bagaria_cn_extendibles_and_vp} in the Choiceless context. Let $\VP(\mathbf{\Pi_{n+1}})$ denote $\boldVP$ restricted to $\Pi_{n+1}$-definable, with parameters, classes of structures.

\newtheorem*{cor:equiconsistency_vp_no_extendibles_with_many_rank_berkeleys}{Corollary~\ref{equiconsistency_vp_no_extendibles_with_many_rank_berkeleys}}
\begin{cor:equiconsistency_vp_no_extendibles_with_many_rank_berkeleys}
    For $n\geq1,$ the following theories are equiconsistent:
    \begin{enumerate}
        \item $\ZF + \boldVP + ``\forall \kappa (\kappa \textrm{\textup{ is not $C^{(0)}$\-/extendible}})"$
        \item $\ZF + \VP(\mathbf{\Pi_{n+1}}) + ``\forall \kappa (\kappa \textrm{\textup{ is not $C^{(n)}$\-/extendible}})"$
        \item $\ZF + \VP(\mathbf{\Pi_{n+1}}) + ``\exists \xi \forall \kappa>\xi (\kappa \textrm{\textup{ is not $C^{(n)}$\-/extendible}})"$
        \item $\ZF + ``\textrm{\textup{There are unboundedly many rank-Berkeley cardinals}}"$
    \end{enumerate}
\end{cor:equiconsistency_vp_no_extendibles_with_many_rank_berkeleys}

Another interesting theory is $\ZF + \boldVP + ``\OR \textrm{\textup{ is not Mahlo}}"$ (Definition~\ref{definition_or_is_mahlo}). Assuming $\AC,$ by using Theorem~\ref{bagaria_cn_extendibles_and_vp}, one can show that $\boldVP$ implies $``\OR \textrm{\textup{ is Mahlo}}".$ In the choiceless context, we have the following:

\newtheorem*{cor:equiconsistency_theories_or_is_not_mahlo}{Corollary~\ref{equiconsistency_theories_or_is_not_mahlo}}
\begin{cor:equiconsistency_theories_or_is_not_mahlo}
    The following theories are equiconsistent:
    \begin{enumerate}
        \item $\ZF + \boldVP + ``\OR \textrm{\textup{ is not Mahlo}}"$
        \item $\ZF + ``\textrm{\textup{There are unboundedly many rank-Berkeley cardinals}}"$
    \end{enumerate}
\end{cor:equiconsistency_theories_or_is_not_mahlo}

The outline of the paper is as follows. In Section~\ref{section_choiceless_extendibles_and_supercompacts}, we introduce ``$n$-choiceless'' extendible and supercompact cardinals, and prove relations between them similar to those of Bagaria's $C^{(n)}$\-/extendible and $\Sigma_n$\-/supercompact cardinals in \cite{cn-cardinals} and \cite{more_on_preservation}. In Section~\ref{section_vopenka_and_choiceless_extendibles}, we prove Theorem~\ref{equivalence_of_vp_and_n_choiceless_extendibles}. In Sections \ref{section_berkeley_and_rank_berkeley_cardinals} and \ref{section_vp_and_berkeley_cardinals}, we define the various notions of Berkeley cardinals, prove some results about them, and show how they relate to Vop\v enka's Principle. Finally, in the last section, we prove Corollaries \ref{equiconsistency_vp_no_extendibles_with_many_rank_berkeleys} and \ref{equiconsistency_theories_or_is_not_mahlo}.


\section{The Choiceless Cardinals}\label{section_choiceless_extendibles_and_supercompacts}

Recall that $C^{(n)}$ is the class of ordinals $\alpha$ that are $\Sigma_n$-correct, i.e., $V_\alpha \prec_{\Sigma_n} V.$ Given any set $X,$ the statement $``X \prec_{\Sigma_n} V"$ is given by the following formula:
\begin{equation}
    \forall \varphi \in \Sigma_n \forall x \in X^{<\omega} (V\satisfies_n \varphi[x] \implies X \satisfies \varphi[x]). \label{prec_sigma_n_formula}
\end{equation}
Now, the satisfaction relation $\satisfies$ for sets is $\Delta_1,$ and the global satisfaction relation $\satisfies_n$ for $\Sigma_n$ formulas is $\Sigma_n,$ for $n\geq 1$ (see \cite{kanamori}, Section 0). Hence, \eqref{prec_sigma_n_formula} is $\Pi_n,$ for $n\geq1.$

The class $C^{(0)}$ is clearly the entire class $\OR$ of ordinals, and is therefore $\Delta_0$-definable. For $n\geq1,$ the class $C^{(n)}$ is defined by:
\begin{equation*}
    \alpha\in C^{(n)} \iff \forall X (X=V_\alpha \implies X\prec_{\Sigma_n} V).
\end{equation*}
Since $``X=V_\alpha"$ is $\Pi_1,$ the defining formula for $C^{(n)}$ is $\Pi_2$ if $n=1,$ and $\Pi_n$ if $n\geq2.$ But, $C^{(1)}$ is also definable using the following:
\begin{equation*}
    \alpha\in C^{(1)} \iff \exists X (X=V_\alpha \wedge X\prec_{\Sigma_1} V),
\end{equation*}
which is $\Sigma_2.$ As a result, we have that $C^{(n)}$ is $\Delta_0$ for $n=0,$ $\Delta_2$ for $n=1,$ and $\Pi_n$ for $n\geq2.$ We remark that if $\AC$ holds, $C^{(1)}$ becomes $\Pi_1$
\cite[Section 1]{cn-cardinals}.

\begin{definition}
    For each $n\geq0,$ given ordinals $\alpha<\gamma<\mu$ with $\mu$ in $C^{(n)},$ we say that $\gamma$ is \emph{$(\alpha,\mu,n)$-choiceless extendible} iff there is $\nu$ in $C^{(n)}$ and an elementary embedding $j:V_\mu \rightarrow V_\nu$ such that $\crit j>\alpha$ and $j(\gamma)>\mu.$
    We say that $\gamma$ is \emph{$\alpha$-$n$-choiceless extendible} iff $\gamma$ is $(\alpha,\mu,n)$-choiceless extendible for all $\mu>\gamma$ in $C^{(n)},$ and we simply say that $\gamma$ is \emph{$n$-choiceless extendible} iff it is $\alpha$-$n$-choiceless extendible for all $\alpha<\gamma.$
\end{definition}

We shall say that $\gamma$ is $(<\!\alpha,\mu,n)$-(respectively $(\leq\!\alpha,\mu,n)$-)choiceless extendible iff it is $(\beta,\mu,n)$-choiceless extendible for all $\beta<\alpha$ (respectively $\beta\leq\alpha).$ Similar remarks hold for $\mu.$ Furthermore, we will allow the occurrence of $\OR$ in the second coordinate. Hence, for example, a cardinal $\gamma$ is $n$-choiceless extendible iff it is $(<\!\gamma,<\!\OR,n)$-choiceless extendible.

The definition above stems from Bagaria's $C^{(n)}$\-/extendible cardinals \cite[12]{cn-cardinals}. Notions somewhat similar to $\alpha$-$0$-choiceless extendibility appear in the works of David Asper\'o \cite{aspero} and Gabriel Goldberg \cite{goldberg_measurable_choiceless}. Notice that if $\gamma$ is $\alpha$-$n$-choiceless extendible, then every ordinal $\delta>\gamma$ is $\alpha$-$n$-choiceless extendible. On the other hand, if $\gamma$ is $n$-choiceless extendible, then it must be a cardinal.

In the context of $\AC,$ one can use the fact that there is no elementary embedding $j:V_{\delta+2}\rightarrow V_{\delta+2}$ to show that an $n$-choiceless extendible cardinal $\gamma$ is either $C^{(n)}$\-/extendible or a limit of $C^{(n)}$\-/extendible cardinals. However, if there is such an embedding, then it is possible that this fails for $\gamma.$ And when it fails, we can argue that the majority of the witnessing $j$s (meaning all except for set many) have critical points strictly between $\alpha$ and $\gamma$ and are such that $\{\beta\mid j(\beta)=\beta\}\cap(\gamma\setminus\crit{j})\neq \emptyset.$ This will be clear in the final section.

A lot of the important properties of $C^{(n)}$\-/extendible cardinals are still provable under this new more general definition, albeit sometimes with slightly more technical difficulties and restrictions. For example, the following is generalized from the case of $C^{(n)}$\-/extendible cardinals in \cite{cn-cardinals}.

\begin{proposition}\label{correctness_of_choiceless_extendibles}
    For each $n\geq 0,$ every $n$-choiceless extendible cardinal is $\Sigma_{n+2}$-correct.
\end{proposition}

Joan Bagaria and Alejandro Poveda \cite{more_on_preservation} prove an equivalence between the notions of $C^{(n)}$\-/extendibility and $\Sigma_{n+1}$\-/supercompactness. An analogous equivalence will be important for our purposes. We are therefore led to the following definitions.

\begin{definition}
    For each $n\geq 0,$ given ordinals $\alpha<\gamma<\lambda$ with $\lambda$ in $C^{(n)},$ and given a set $a\in V_\lambda,$ we say that $\gamma$ is \emph{$(\alpha,\lambda, a, n)$-choiceless supercompact} iff there exists $\bar \lambda<\gamma$ in $C^{(n)}$ and $\bar a \in V_{\bar \lambda}$ for which there is an elementary embedding $j:V_{\bar\lambda}\rightarrow V_\lambda$ with $\crit j>\alpha$ and $j(\bar a)=a.$
    We say that $\gamma$ is \emph{$\alpha$-$n$-choiceless supercompact} iff $\gamma$ is $(\alpha,\lambda, a, n)$-choiceless supercompact for all $\lambda>\gamma$ in $C^{(n)}$ and for all $a\in V_\lambda,$ and we simply say that $\gamma$ is $n$-choiceless supercompact iff it is $\alpha$-$n$-choiceless supercompact for all $\alpha<\gamma.$
\end{definition}

Again, similar to the case of choiceless extendible cardinals, we allow the use of inequality symbols and $\OR$ in our notation for choiceless supercompact cardinals. Here, for the third coordinate, we will also use the notation $<\!X,$ where it will mean that the set $a$ can be any member of the set or class $X.$

\begin{definition}
    For each $n\geq0,$ given ordinals $\alpha<\gamma<\mu,$ we say that $\gamma$ is \emph{$(\alpha,\mu,n)$-choiceless extendible$^*$} iff there is a $\nu>\mu$ and an elementary embedding $j:V_\mu \rightarrow V_\nu$ such that $\crit j>\alpha,\ j(\gamma)>\mu,$ and $j(\gamma)\in C^{(n)}.$
    We say that $\gamma$ is \emph{$\alpha$-$n$-choiceless extendible$^*$} iff $\gamma$ is $(\alpha,\mu,n)$-choiceless extendible$^*$ for all $\mu>\gamma$ in $C^{(n)},$ and we simply say that $\gamma$ is \emph{$n$-choiceless extendible$^*$} iff it is $(\alpha,\mu,n)$-choiceless extendible$^*$ for all $\alpha<\gamma.$
\end{definition}

Remarks similar to those following the previous two definitions about the use of symbols such as $\leq$ and $\OR$ apply for the above definition as well. The following is easy to show.

\begin{proposition}\label{correctness_of_star_choiceless_extendibles}
    For each $n\geq 0,$ every $n$-choiceless extendible$^*$ cardinal is $\Sigma_{n+2}$-correct.
\end{proposition}

\begin{lemma}\label{star_choiceless_extendible_to_choiceless_supercompact}
    For $n\geq0,$ if $\gamma$ is $\alpha$-$n$-choiceless extendible$^*$ for some fixed $\alpha<\gamma,$ then it is also $\alpha$-$n+1$-choiceless supercompact.
\end{lemma}

\begin{proof}
    Suppose $\gamma$ is $\alpha$-$n$-choiceless extendible$^*.$ Fix $\lambda>\gamma$ in $C^{(n+1)}$ and a set $a\in V_\lambda,$ and let us show that $\gamma$ is $(\alpha,\lambda,a,n+1)$-choiceless supercompact. Let $\mu>\lambda$ be in $C^{(n+1)}$ and, using the fact that $\gamma$ is $\alpha$-$n$-choiceless extendible$^*,$ let $j:V_\mu\rightarrow V_\nu$ be such that $j(\gamma)>\mu,$ where $j(\gamma)$ is in $C^{(n)},$ and $\crit j>\alpha.$ Notice now that $j\vert_{V_\lambda}$ belongs to $V_\nu.$ 
    
    \begin{claim}
        $V_\lambda\prec_{\Sigma_{n+1}}V_\nu.$
    \end{claim}
    
    \begin{proof}[Proof of claim]
        On the one hand, $\lambda\in C^{(n+1)}$ and $j(\gamma)\in C^{(n)}$ imply $V_\lambda\prec_{\Sigma_{n+1}}V_{j(\gamma)}.$ On the other hand, $V_\gamma\prec_{\Sigma_{n+1}}V_\mu$ by Proposition~\ref{correctness_of_star_choiceless_extendibles}, hence elementarity of $j$ gives $V_{j(\gamma)}\prec_{\Sigma_{n+1}}V_\nu.$ Putting both together, the claim follows.
    \end{proof}
    
    Thus, $j\vert_{V_\lambda}$ witnesses in $V_\nu$ the $(\alpha,j(\lambda),j(a),n+1)$-choiceless supercompactness of $j(\gamma).$ By elementarity of $j,$ there must be some $k$ witnessing the $(\alpha,\lambda,a,n+1)$-choiceless supercompactness of $\gamma$ in $V_\mu.$ But, since $\mu$ is correct enough, any such $k$ will be a real witness for the $(\alpha,\lambda,a,n+1)$-choiceless supercompactness of $\gamma.$ Since $\lambda$ and $a$ were arbitrary, we are done.
\end{proof}

\begin{lemma}\label{choiceless_supercompact_to_choiceless_extendible}
    For $n\geq1,$ if $\gamma$ is $\alpha$-$n+1$-choiceless supercompact for some fixed $\alpha<\gamma,$ then it is also $\alpha$-$n$-choiceless extendible.
\end{lemma}

\begin{proof}
    Suppose $\gamma$ is $\alpha$-$n+1$-choiceless supercompact, and fix $\mu>\gamma$ in $C^{(n)}.$ We want to show that $\gamma$ is $(\alpha,\mu,n)$-choiceless extendible. Let $\lambda>\mu$ be in $C^{(n+1)}$ and using the $\alpha$-$n+1$-choiceless supercompactness of $\gamma,$ let $j:V_{\bar\lambda}\rightarrow V_{\lambda}$ be an elementary embedding such that $\crit j>\alpha,\ j(\bar \gamma)=\gamma, \ j(\bar \mu)=\mu,$ and $\bar\lambda\in C^{(n+1)}.$ Noticing, by elementarity, that $\bar \mu$ is in fact in $C^{(n)},$ we conclude that $j\vert_{V_{\bar\mu}}$ witnesses the $(\alpha,\bar\mu,n)$-choiceless extendibility of $\bar \gamma.$ The existence of such witness is a $\Sigma_{n+1}$ statement in the parameters $\alpha, \bar \mu, V_{\bar \mu},$ and $\bar\gamma,$ as seen in the formula
    \begin{equation}
        \exists k, Y, \nu (Y=V_\nu \wedge Y\prec_{\Sigma_n} V \wedge k:V_{\bar\mu}\prec Y \wedge \crit k>\alpha \wedge k(\bar\gamma)>\bar\mu).\label{choiceless_super_to_ext_displayed_formula}
    \end{equation}
    Since $\bar\lambda\in C^{(n+1)},$ we must have $k\in V_{\bar\lambda}$ that witnesses the $(\alpha,\bar\mu,n)$-choiceless extendibility of $\bar \gamma$ in $V_{\bar \lambda}.$ Now elementarity of $j$ tells us that $j(k)$ witnesses the $(\alpha,\mu,n)$-choiceless extendibility of $\gamma$ in $V_{\lambda}.$ Since $\lambda$ is correct enough, this last statement is also true in $V;$ and as $\lambda$ was arbitrary, we are done.
\end{proof}

Note that, for $n=0,$ the proof above fails because the complexity of \eqref{choiceless_super_to_ext_displayed_formula} is $\Sigma_2$ rather than $\Sigma_1.$ In fact, it is easy to show that proving the case $n=0$ will render supercompact cardinals inconsistent in $\ZFC.$


\begin{theorem}\label{equivalence_of_choiceless_extendibles_and_choiceless_supercompacts}
    For $n\geq 1,$ the following are equivalent:
    \begin{enumerate}
        \item $\gamma$ is $n$-choiceless extendible.
        \item $\gamma$ is $n$-choiceless extendible$^*.$
        \item $\gamma$ is $n+1$-choiceless supercompact.
    \end{enumerate}
\end{theorem}

\begin{proof}
    (i) implies (ii): Clear using Proposition~\ref{correctness_of_choiceless_extendibles}.
    (ii) implies (iii): Lemma~\ref{star_choiceless_extendible_to_choiceless_supercompact}.
    (iii) implies (i): Lemma~\ref{choiceless_supercompact_to_choiceless_extendible}.
\end{proof}


\section{VP and Choiceless Extendible Cardinals}\label{section_vopenka_and_choiceless_extendibles}

Recall that \emph{Vop\v enka's Principle} is the axiom schema stating that for every proper class $\mathcal{C}$ of structures of the same type that is definable, with parameters, there exist $A\neq B$ in $\mathcal{C}$ such that $A$ is elementarily embeddable into $B.$

Say that a class $\mathcal{C}$ is $\Sigma_n(X)$ (respectively $\Pi_n(X)$) for some class (or set) $X$ iff $\mathcal{C}$ is definable with parameters from $X$ by a $\Sigma_n$ (respectively $\Pi_n$) formula. The boldface symbols $\mathbf{\Sigma_n}$ and $\mathbf{\Pi_n}$ are used in place of $\Sigma_n(V)$ and $\Pi_n(V),$ respectively, and the lightface symbols $\Sigma_n$ and $\Pi_n$ are used in place of $\Sigma_n(\emptyset)$ and $\Pi_n(\emptyset),$ respectively.

Let $\Gamma$ be a placeholder for the symbols $\mathbf{\Sigma_n}, \mathbf{\Pi_n}, \Sigma_n, \Pi_n, \Sigma_n(X), \Pi_n(X).$ We say that $\VP(\Gamma)$ holds iff Vopenka's Principle holds for any proper class $\mathcal{C}$ (of structures of the same type) that is $\Gamma.$ The notation $\VP(X)$ for a class of parameters $X$ is used iff $\VP(\Pi_n(X))$ holds for all $n.$
We will also use $\VP$ and $\boldVP$ in place of $\VP(\emptyset)$ and $\VP(V),$ respectively.

The failure of $\VP(\Gamma)$ will be denoted by $\neg\VP(\Gamma).$ Notice that, although $\VP(\Gamma)$ is an axiom schema, $\neg\VP(\Gamma)$ can be stated by a single axiom.

\begin{theorem}\label{equivalence_of_vp_and_n_choiceless_extendibles}
    For $n\geq1,$ the following are equivalent:
    \begin{enumerate}
        \item $\VP(\Pi_{n+1}(\OR)).$
        \item There is a proper class of $n$-choiceless extendible cardinals.
        \item $\VP(\mathbf{\Pi_{n+1}}).$
    \end{enumerate}
\end{theorem}

\begin{proof}
(i) implies (ii): We will show that $\VP(\Pi_{n+1}(\OR))$ implies that, for any $\alpha,$ there is a $\gamma_\alpha>\alpha$ that is $\alpha$-$n$-choiceless extendible. Given this, we can then construct the sequence $\langle \alpha_n\rangle_{n\in\omega}$ by letting $\alpha_0=\alpha$ and $\alpha_{n+1}=\gamma_{\alpha_n}.$ Clearly then $\delta=\lim \alpha_n$ will be a proper $n$-choiceless extendible cardinal. Moreover, since $\delta>\alpha$ and $\alpha$ was arbitrary, it will follow that there is a proper class of $n$-choiceless extendible cardinals.

So, let $\alpha$ be an arbitrary fixed ordinal and suppose, towards a contradiction, that no ordinal $\gamma>\alpha$ is $\alpha$-$n$-choiceless extendible. This means that for any ordinal $\gamma>\alpha$ there exists $\mu>\gamma$ in $C^{(n)}$ for which the following holds:
$$
\neg \exists j:V_\mu\prec V_\nu(\crit j > \alpha \wedge j(\gamma)>\mu \wedge V_\nu\prec_{\Sigma_n} V).
$$
Let $\psi(\alpha,\gamma,\mu)$ the displayed formula above. Its complexity, for $n\geq1,$ is $\Pi_{n+1}.$ Define $D(\alpha,n)$ to be the set of all $\beta$ such that $\beta$ is a limit ordinal above $\alpha$ and for every $\gamma$ strictly between $\alpha$ and $\beta$ there already is a $\mu\in C^{(n)}$ also below $\beta$ for which $\psi(\alpha,\gamma,\mu)$ holds. Formally, 
\begin{multline*}
    \beta\in D(\alpha,n)\iff \beta\in(\operatorname{Lim}(\OR)-(\alpha+1)) \\
        \wedge \forall \gamma(\alpha<\gamma<\beta) \exists\mu < \beta (\mu\in C^{(n)}\wedge\psi(\alpha,\gamma,\mu)).
\end{multline*}
It is easy to see that $D(\alpha,n)$ is a club subclass of $\operatorname{Lim}(C^{(n)})$ that is $\Pi_{n+1}$-definable. Now let $\mathcal{C}$ be the class of structures of the form $(V_\beta, D(\alpha,n)\cap\beta, \xi)_{\xi\in\alpha}$ such that $\beta\in D(\alpha,n)$ and $D(\alpha,n)\cap\beta\in V_\beta$ (this last condition ensures that $D(\alpha,n)\cap\beta$ is bounded below $\beta$ so that $\beta$ is not a limit of ordinals in $D(\alpha,n)).$

The class $\mathcal{C}$ is $\Pi_{n+1}(\OR)$ so we apply $\VP(\Pi_{n+1}(\OR))$ and get an elementary embedding
$$
j:(V_{\beta_1}, D(\alpha,n)\cap\beta_1,\xi)_{\xi\leq\alpha}\rightarrow (V_{\beta_2},D(\alpha,n)\cap\beta_2,\xi)_{\xi\leq\alpha}
$$
where $\beta_1\neq\beta_2.$ Let $\sigma_i=\sup(D(\alpha,n)\cap\beta_i)$ for $i=1,2.$  Notice that the $\sigma_i$ are both in $D(\alpha,n)$ and that $j(\sigma_1)=\sigma_2.$ Since each $\sigma_i$ is uniquely identified by their respective $\beta_i$ and since $\beta_1\neq \beta_2,$ we also have $\sigma_1\neq \sigma_2.$ In particular, $j$ is not the identity. We have $j(\xi)=\xi$ for all $\xi\leq \alpha$ due to the constants $\xi$ for all $\xi\leq\alpha,$ so the critical point of $j$ must be above $\alpha.$  As $\beta_1\in D(\alpha,n)\subset \operatorname{Lim}(C^{(n)}),$ there are $\mu\in C^{(n)}$ arbitrarily high in $\beta_1.$ For any such $\mu,\ j(\mu)$ is in $C^{(n)}$ as well, by elementarity of $j:V_{\beta_1}\rightarrow V_{\beta_2}$ and the fact that the $\beta_i$ are both themselves in $C^{(n)}.$ Thus, the restriction of $j$ to $V_\mu$ for any $\mu\in C^{(n)}$ with $\sigma_1<\mu<\beta_1$ will give us an elementary embedding $j\vert_{V_\mu}:V_\mu\rightarrow V_{j(\mu)}$ with $\crit(j\vert_{V_\mu}) >\alpha,\ j\vert_{V_\mu}(\sigma_1)=\sigma_2\geq \beta_1>\mu,$ and $j(\mu)\in C^{(n)}.$ But this cannot be possible since we know that, as $\beta_1\in D(\alpha,n),$ for some $\mu\in C^{(n)}$ with $\sigma_1<\mu<\beta_1$ the formula $\psi(\alpha,\sigma_1,\mu)$ holds, hence such an elementary embedding cannot exist.

(ii) implies (iii): Let $\mathcal{C}$ be a proper class of structures of the same type $\tau$ that is $\mathbf{\Pi_{n+1}}.$ Let $\gamma$ be $n$-choiceless extendible and sufficiently large so that there exists some $\alpha<\gamma$ such that $\tau$ along with any parameters $p$ of some defining $\Pi_{n+1}$ formula for $\mathcal{C}$ are all in $V_\alpha.$ Fix such an $\alpha.$ Using Theorem~\ref{equivalence_of_choiceless_extendibles_and_choiceless_supercompacts} we know that $\gamma$ is $n+1$-choiceless supercompact. Let $B\in \mathcal{C}$ have rank above $\gamma$ and let $\lambda\in C^{(n+1)}$ be an ordinal above this rank. By $n+1$-choiceless supercompactness let $j:V_{\bar\lambda}\rightarrow V_\lambda$ be an elementary embedding with $\bar\lambda\in C^{(n+1)},\ j(\bar B)=B,$ and $\crit j>\alpha.$ By correctness of $\bar\lambda$ and elementarity of $j$ we must have $\bar B\in \mathcal{C},$ and obviously $\bar B\neq B$ by considering their respective ranks. Hence, the restriction of $j$ to $\bar B$ is an elementary embedding from $\bar B$ into $B,$ and we are done.

(iii) implies (i): Trivial.
\end{proof}

\begin{corollary}\label{vp_to_boldvp}
    The following are equivalent:
        \item \textup{(i)} $\VP,$ \textup{(ii)} $\VP(\OR),$ and \textup{(iii)} $\boldVP.$
\end{corollary}

\begin{proof}
    (i) implies (ii): Let $\phi(x,\alpha)$ be a formula that defines, for some ordinal $\alpha,$ a class of structures for which $\VP(\OR)$ fails. Notice that the least $\alpha$ for which this happens is definable without parameters. Thus, $\VP$ fails too.

    (ii) implies (iii) by the previous theorem, and (iii) implies (i) trivially.
\end{proof}


\section{Berkeley Cardinals}\label{section_berkeley_and_rank_berkeley_cardinals}

We start by recalling the definition of and some basic facts about Berkeley cardinals from \cite{lcbc}. 

\begin{definition}[\cite{lcbc}]\label{def:berkeley}
    A cardinal $\delta$ is \emph{$\zeta$-proto Berkeley,} for some ordinal $\zeta<\delta,$ iff for all transitive $M$ with $\delta\in M$ there is an elementary embedding $j:M\rightarrow M$ with $\zeta<\crit{j}<\delta.$ A cardinal $\delta$ is \emph{Berkeley} iff it is $\zeta$-proto Berkeley for all $\zeta<\delta.$
\end{definition}

\begin{lemma}[\cite{lcbc}]\label{proto_berkeley_fix_any_b_trick}
    A cardinal $\delta$ is $\zeta$-proto Berkeley iff for all transitive $M$ with $\delta\in M$ and all $b\in M$ there is an elementary embedding $j:M\rightarrow M$ with $\zeta<\crit{j}<\delta$ and $j(b)=b.$
\end{lemma}

\begin{proposition}[\cite{lcbc}]\label{berkeley_from_proto}
    For any fixed ordinal $\zeta,$ the least $\zeta$-proto Berkeley cardinal, if it exists, is also a Berkeley cardinal.
\end{proposition}

For our purposes, we will also be interested in a somewhat weakened version of Berkeley cardinals. We will simply restrict the definitions to transitive sets $M$ of the form $V_\lambda$ for some $\lambda.$

\begin{definition}\label{def:rank_berkeley}
    A cardinal $\delta$ is \emph{$\zeta$-proto rank-Berkeley,} for some ordinal $\zeta<\delta,$ iff for all $\lambda>\delta$ there is an elementary embedding $j:V_\lambda \rightarrow V_\lambda$ with $\zeta<\crit{j}<\delta$ and $j(\delta)=\delta.$ A cardinal $\delta$ is \emph{rank-Berkeley} iff it is $\zeta$-proto rank-Berkeley for all $\zeta<\delta.$
\end{definition}

Let $\mathcal{E}_\lambda$ denote the set of all nontrivial elementary embeddings $j:V_\lambda\rightarrow V_\lambda$ and let $\mathcal{E}_\lambda^\delta=\{j\in\mathcal{E}_\lambda\mid \crit{j}<\delta \text{ and } j(\delta)=\delta\}.$
We want to have an analogue of Lemma~\ref{proto_berkeley_fix_any_b_trick} which will allow us to impose the extra condition of fixing an arbitrary ordinal $\alpha\in V_\lambda$ on $j:V_\lambda\rightarrow V_\lambda.$ For this, we first need to define an operation from $\mathcal{E}_\lambda^\delta \times \mathcal{E}_\lambda^\delta$ to $\mathcal{E}_\lambda^\delta.$

\begin{definition}
    If $\lambda$ is a limit ordinal, then for any $j,k:V_\lambda\rightarrow V_\lambda$ define the operation $j[k],$ the \emph{application of $j$ to $k,$} by setting $j[k] = \bigcup_{\gamma<\lambda}j(k\vert_{V_\gamma}).$
\end{definition}

The following lemma is similar to Lemma~1.6 in \cite{dehornoy}.

\begin{lemma}\label{application_op_and_critical_point}
    If $\lambda$ is a limit and $j,k\in \mathcal{E}_\lambda^\delta,$ then $j[k]$ is also in $\mathcal{E}_\lambda^\delta.$ Moreover, $\crit j[k] = j(\crit{k}).$
\end{lemma}

Now, for limit $\lambda$ and $j\in\mathcal{E}_\lambda^\delta,$ define $j_0=j$ and $j_{n+1}=j[j_n].$ By induction and the previous lemma, each $j_n$ is in $\mathcal{E}_\lambda^\delta.$ The lemma below, which is due to Schlutzenberg, will now work as an analogue of Lemma~\ref{proto_berkeley_fix_any_b_trick}.

\begin{lemma}[{\cite[Lemma~1.3]{schlutzenberg}}]\label{iterate_to_fix}
    For limit $\lambda,$ any $j\in \mathcal{E}_\lambda^\delta,$ and each $\alpha\in V_\lambda$ there is $n$ such that $j_m(\alpha)=\alpha$ for all $m\geq n.$
\end{lemma}

\begin{proposition}\label{rank_berkeley_from_proto}
    For any ordinal $\zeta,$ the least $\zeta$-proto rank-Berkeley cardinal, if it exists, is also a rank-Berkeley cardinal.
\end{proposition}

\begin{proof}
    Let $\delta$ be the least $\zeta$-proto rank-Berkeley cardinal and suppose it is not rank-Berkeley. Fix $\alpha\in(\zeta,\delta)$ and $\lambda>\delta$ such that $\forall j\in \mathcal{E}_{\lambda}^{\delta} (\crit{j}\leq\alpha).$ We will show that $\alpha$ must be a $\zeta$-proto rank-Berkeley, contradicting the choice of $\delta.$

    Let $\mu>\alpha$ be arbitrary and let $\nu$ be an ordinal above $\max\{\mu, \lambda\}.$ By $\zeta$-proto rank-Berkeleyness of $\delta,$ fix a $j\in \mathcal{E}_\nu^{\delta}$ such that $\crit j>\zeta.$ By Lemma~\ref{iterate_to_fix}, there is $n$ such that $j_n$ fixes $\alpha,\lambda,$ and $\mu,$ and moreover, $\crit j_n >\zeta.$ Now, $j_n\vert_{V_{\lambda}}\in \mathcal{E}_{\lambda}^{\delta},$ hence $\crit{j_n} \leq \alpha.$ But also $j_n(\alpha)=\alpha,$ so in fact $\crit{j_n}<\alpha.$ Finally, since $j_n$ fixes $\mu,$ we can restrict $j_n$ to $V_\mu$ so that $j_n\vert_{V_\mu}$ witnesses $\zeta$-proto rank-Berkeleyness of $\alpha$ at the arbitrary ordinal $\mu.$
\end{proof}


\section{VP and Berkeley Cardinals}\label{section_vp_and_berkeley_cardinals}

Let $\VP^\alpha(X)$ be $\VP(X)$ restricted to structures of type $\tau\in V_\alpha.$ The following result is similar to \cite[Corollary~2.3]{goldberg_measurable_choiceless}, but for Berkeley cardinals.

\begin{proposition}\label{berkeley_to_vp}
    If $\delta$ is a Berkeley cardinal, then $\boldVP^\delta$ holds.
\end{proposition}

\begin{proof}
    Let $\mathcal{C}$ be a definable, with parameters, proper class of structures of the same type $\tau\in V_\delta.$ Let $\rank(x)$ denote the rank function and $\ot(x)$ the order-type function on sets of ordinals. Consider the class function $F: \mathcal{C}\rightarrow \OR$ defined by 
    $$
    F(A) = \ot(\rank(A)+1\cap\range(\operatorname{rank}\vert_{\mathcal{C}})).
    $$
    Since $\mathcal{C}$ is a proper class, $\range(F)=\OR.$

    Denote $\mathcal{C}\cap V_\alpha$ by $\mathcal{C}_\alpha.$ Let $\lambda>\delta$ be large enough so that $\delta\subset \range(F\vert_{\mathcal{C}_\lambda}).$ By Berkeleyness, fix an elementary embedding $j:V_{\lambda+1}\rightarrow V_{\lambda+1}$ such that $j(\mathcal{C}_\lambda)=\mathcal{C}_\lambda$ and $\tau\in V_\kappa,$ where $\kappa$ is the critical point. We know that there is $A\in V_\lambda$ such that $F(A)=\kappa.$ Now, since $F(j(A))=j(F(A))=j(\kappa)\neq \kappa,$ we must have $A\neq j(A).$ Also, as $j$ fixes $\mathcal{C}_\lambda,\ j(A)$ is in $\mathcal{C}_\lambda$ as well. So $j$ restricted to the structure $A$ of $\mathcal{C}$ elementarily embeds it into the different structure $j(A)$ of $\mathcal{C},$ and we are done.
\end{proof}

In the other direction, by building on the proof of Theorem~\ref{equivalence_of_vp_and_n_choiceless_extendibles}, we get:

\begin{theorem}[$\ZF + \neg\boldVP$]\label{vp_to_berkeley_levels}
    If $\boldVP^\omega$ holds, then there is an ordinal $\delta>\omega$ for which $\boldVP^\delta$ holds while $\boldVP^{\delta+1}$ fails, and moreover, $\delta$ is a Berkeley cardinal.
\end{theorem}

\begin{proof}
    Since $\boldVP$ fails, by Theorem~\ref{equivalence_of_vp_and_n_choiceless_extendibles}, for some $n\geq1,$ there are only boundedly many $n$-choiceless extendible cardinals, if any. By the first paragraph of the proof of \ref{equivalence_of_vp_and_n_choiceless_extendibles}, for some $\alpha,$ there is no $\alpha$-$n$-choiceless extendible ordinal above $\alpha.$ We notice that, by the definition of $\gamma$-$n$-choiceless extendibility, all $\alpha>\gamma$ also satisfy that there is no $\alpha$-$n$-choiceless extendible ordinal above $\alpha.$

    Define the class $\mathcal{C}(\alpha,M,\zeta)$ (with arbitrary parameters $\alpha,M,\zeta)$ to be the class of structures of the form $(V_\beta,D(\alpha,n)\cap\beta,\alpha,M,\xi)_{\xi\leq\zeta}$ where $D(\alpha,n)$ is the same as in Theorem~\ref{equivalence_of_vp_and_n_choiceless_extendibles}, $\beta\in D(\alpha,n),\ D(\alpha,n)\cap\beta\in V_\beta,$ and $M\in V_\beta.$ Let $m$ be a natural number large enough so that $\mathcal{C}(\alpha,M,\zeta)$ is $\Sigma_m(V)$ and, moreover, $\VP(\Sigma_m(V))$ fails. Note that $m$ does not change regardless of the choice of $\zeta.$ 
    
    Since $\VP^\omega(\Sigma_m(V))$ holds while $\VP(\Sigma_m(V))$ fails, define $\delta\geq\omega$ to be $\sup\{\mu\mid \VP^\mu(\Sigma_m(V)) \textrm{ holds}\},$ and let $\alpha$ be some cardinal greater than both $\gamma$ and $\delta.$ We will show that $\alpha$ is a $\zeta$-proto Berkeley cardinal for all $\zeta<\delta.$ By Proposition~\ref{berkeley_from_proto}, this implies that, for each $\zeta<\delta,$ there is a Berkeley cardinal $\delta_\zeta$ such that $\zeta<\delta_\zeta\leq\alpha.$ By the definition of $\delta$ and Proposition~\ref{berkeley_to_vp}, none of the $\delta_\zeta$s can be greater than $\delta.$ In particular, this means that $\delta$ will either be $\delta_\zeta$ for some $\zeta$ or the limit of the $\delta_\zeta$s. In both cases, $\delta$ will be a Berkeley cardinal. Then, again using Proposition~\ref{berkeley_to_vp}, we must have $\boldVP^\delta,$ and, by the definition of $\delta,\ \boldVP^{\delta+1}$ must fail. Finally, $\delta$ being Berkeley means $\delta>\omega,$ and we will have found our $\delta.$

    So all we need to do is to show that $\alpha$ is a $\zeta$-proto Berkeley for any $\zeta<\delta.$ Fix $\zeta<\delta.$ Let $M$ be any transitive set that contains $\alpha.$ An application of $\VP^\delta(\Sigma_m(V))$ to the class $\mathcal{C}(\alpha, M, \zeta)$ will give us an elementary embedding
    $$
    j:(V_{\beta_1},D(\alpha,n)\cap\beta_1,\alpha,M,\xi)_{\xi\leq\zeta}\rightarrow (V_{\beta_2},D(\alpha,n)\cap\beta_2,\alpha,M,\xi)_{\xi\leq\zeta}
    $$
    for some $\beta_1\neq \beta_2.$
    $j$ is not the identity, because $\beta_1\neq\beta_2$ implies $j(\sigma_1)=\sigma_2\neq\sigma_1,$ where $\sigma_i=\sup(D(\alpha,n)\cap\beta_i),\ i\in\{1,2\}.$ The critical point of $j$ cannot be $\alpha$ as $\alpha$ is fixed by $j,$ and it cannot be higher than $\alpha$ either by the definition of $\mathcal{C}.$ Moreover, the critical point has to be strictly greater than $\zeta$ due to the constants $\xi\leq\zeta.$ So, we are left with $\zeta<\crit j<\alpha.$ Thus, the restriction of $j$ to $M$ is our desired elementary embedding.
\end{proof}

\begin{corollary}[$\ZF + \neg\boldVP$]\label{cor:vp_to_berkeley_levels_sup}
    If $\boldVP^\omega$ holds, then
    $$
    \sup\{\mu\mid \boldVP^\mu\textrm{ holds}\}=\sup\{\delta\mid \delta \text{ is Berkeley}\}.
    $$
\end{corollary}

\begin{proof}
    By the proof of Theorem~\ref{vp_to_berkeley_levels} and Proposition~\ref{berkeley_to_vp}.
\end{proof}

\begin{corollary}[$\ZF + \neg\boldVP$]\label{berkeley_equivalent_to_finite_vp}
    The existence of a Berkeley cardinal is equivalent to $\boldVP^\omega.$
\end{corollary}

\begin{corollary}\label{berkeley_weak_vp_to_strong_vp}
    If $\AC$ holds, then $\boldVP^\omega$ implies $\boldVP.$
\end{corollary}

\begin{remark}
    The proof of Theorem~\ref{vp_to_berkeley_levels} actually shows that $\VP^\omega(\mathbf{\Pi}_{n+1})$ implies the existence of a Berkeley cardinal in $\ZF + \neg \VP(\mathbf{\Pi}_{n+1}),$ for $n\geq1.$ With that in mind, Corollary~\ref{berkeley_equivalent_to_finite_vp} becomes an equivalence between the existence of a Berkeley cardinal, $\boldVP^\omega,$ and $\VP^\omega(\mathbf{\Pi}_{n+1})$ in $\ZF + \neg \VP(\mathbf{\Pi}_{n+1}),$ for $n\geq1.$
\end{remark}

Similar relations hold between rank-Berkeley cardinals and $\VP(\OR).$ For example, in the proof of Proposition~\ref{berkeley_to_vp}, if $\lambda$ is chosen to be correct enough and $j$ is so that it fixes the ordinal defining the class of structures, then we would have:

\begin{proposition}[{\cite[Cor.~2.3]{goldberg_measurable_choiceless}}]\label{rank_berkeley_to_vp}
    If $\delta$ is a rank-Berkeley cardinal, then $\VP^\delta(\OR)$ holds.
\end{proposition}

An analogue of Theorem~\ref{vp_to_berkeley_levels} for rank-Berkeley cardinals is achieved in an essentially similar way. The only detail that is not outright obvious is the fact that a limit of rank-Berkeley cardinals is again a rank-Berkeley cardinal. In the case of Berkeley cardinals, this is straightforward from the definition. But, for rank-Berkeley cardinals, Lemma~\ref{iterate_to_fix} is necessary. Thus, let $\delta$ be a limit of rank Berkeley cardinals and let $\alpha, \lambda$ be arbitrary ordinals satisfying $\alpha<\delta<\lambda.$ Fix a rank-Berkeley cardinal $\delta_0$ such that $\alpha<\delta_0<\delta,$ and let $j\in\mathcal{E}_\lambda^{\delta_0}$ be such that $\crit k>\alpha.$ Lemma~\ref{iterate_to_fix} is now necessary to find a $j_n\in\mathcal{E}_\lambda^{\delta_0}$ that fixes $\delta,$ so that $j_n\in\mathcal{E}_\lambda^\delta.$

\begin{theorem}[$\ZF + \neg\boldVP$]\label{vp_to_rank_berkeley_levels}
    If $\VP^\omega(\OR)$ holds, then there is an ordinal $\delta>\omega$ for which $\VP^\delta(\OR)$ holds while $\VP^{\delta+1}(\OR)$ fails, and moreover, $\delta$ is a rank-Berkeley cardinal.
\end{theorem}

Recall that $\boldVP$ and $\VP(\OR)$ are equivalent by Corollary~\ref{vp_to_boldvp}, so the two background theories $\ZF + \neg\boldVP$ and $\ZF + \neg\VP(\OR)$ are the same.

\begin{corollary}[$\ZF + \neg\boldVP$]
    If $\VP^\omega(\OR)$ holds, then
    $$
    \sup\{\mu\mid \VP^\mu(\OR)\textrm{ holds}\}=\sup\{\delta\mid \delta \text{ is rank-Berkeley}\}.
    $$
\end{corollary}

\begin{corollary}[$\ZF + \neg\boldVP$]
    The existence of a rank-Berkeley cardinal is equivalent to $\VP^\omega(\OR).$
\end{corollary}

\begin{corollary}\label{rank_berkeley_weak_vp_to_strong_vp}
    If $\AC$ holds, then $\VP^\omega(\OR)$ implies $\VP(\OR).$
\end{corollary}

This last corollary is a much stronger result than the analogous Corollary~\ref{berkeley_weak_vp_to_strong_vp}. We already know that $\VP(\OR)$ is equivalent to $\boldVP$ (Corollary~\ref{vp_to_boldvp}). By using the proof of the case (i) implies (ii) of \ref{vp_to_boldvp}, we can also show that $\VP^\omega$ implies $\VP^\omega(\OR).$ Putting everything together, we now have the following:

\begin{corollary}\label{weakest_vp_to_strongest_vp}
    If $\AC$ holds, then $\VP^\omega$ implies (and hence is equivalent to) $\boldVP.$
\end{corollary}

In other words, assuming $\AC,$ the weakest form of Vop\v enka's Principle, where we only allow for definable, with no parameters, proper classes of structures of the same finite type, implies the strongest form of Vop\v enka's Principle, where we allow for all definable, with parameters, proper classes of structures of any type.

Let $\BC$ and $\rBC$ denote the axioms asserting the existence of a Berkeley cardinal and a rank-Berkeley cardinal, respectively. We show next that the theories $\ZF+\boldVP^\omega+\neg\boldVP$ and $\ZF+\BC$ are equiconsistent. A cardinal $\kappa$ is \emph{inaccessible} iff there is no cofinal map $f:V_\alpha\rightarrow \kappa,$ for any $\alpha<\kappa.$ It is easy to show that if $\kappa$ is inaccessible then $(V_\kappa,V_{\kappa+1})\satisfies \ZF_2,$ where $\ZF_2$ is the second-order Zermelo-Fraenkel set theory. Moreover, the critical point of any nontrivial elementary embedding $j:V_\mu\rightarrow V_\nu$ is an inaccessible cardinal.

\begin{theorem}\label{equiconsistency_theories_berkeley_and_finite_vp}
    The theory $\ZF+\BC$ is equiconsistent with $\ZF+\boldVP^\omega+\neg\boldVP.$
\end{theorem}

\begin{proof}
    Theorem~\ref{vp_to_berkeley_levels} deals with one direction. For the other direction, assume, for a contradiction, that $\ZF+\BC$ is consistent while $\ZF+\boldVP^\omega+\neg\boldVP$ is not. That means every model of the former is also a model of $\boldVP$ (by Proposition~\ref{berkeley_to_vp}). In particular, all models of $\ZF+\BC$ are models of $\VP(\Pi_1(V)).$ By G\"odel's completeness theorem, this means that there is a proof of $\VP(\Pi_1(V))$ from $\ZF+\BC.$ But, the theory $\ZF+\BC+\VP(\Pi_1(V))$ proves $\con{\ulcorner\ZF+\BC\urcorner}$ as shown in the next paragraph. This contradicts G\"odel's second incompleteness theorem.

    To show that $\ZF+\BC+\VP(\Pi_1(V))$ proves $\con{\ulcorner\ZF+\BC\urcorner},$ let $\delta$ be a Berkeley cardinal and consider the $\Pi_1(V)$ class $\mathcal{C}=\{(V_{\alpha+1},\xi)_{\xi\leq\delta}\mid \alpha>\delta\}.$ By $\VP(\Pi_1(V)),$ fix an elementary embedding $j:(V_{\alpha_1+1},\xi)_{\xi\leq\delta}\rightarrow (V_{\alpha_2+1},\xi)_{\xi\leq\delta},$ where $\alpha_1\neq \alpha_2.$ Notice that $j(\alpha_1)=\alpha_2$ and $j(\xi)=\xi$ for all $\xi\leq\delta.$ Thus, $j$ has a critical point $\kappa>\delta,$ and therefore $V_\kappa\satisfies\ulcorner\ZF+\BC\urcorner.$
\end{proof}

The case of rank-Berkeley cardinals is as follows:

\begin{theorem}
    The theory $\ZF+\rBC$ is equiconsistent with $\ZF+\VP^\omega(\OR)+\neg\VP(\OR).$
\end{theorem}

\section{Class Many Rank-Berkeley Cardinals}\label{section_equiconsistency_or_not_mahlo_theory}

In this section, we will consider the failure in $\ZF$ of two results; Theorem~\ref{bagaria_cn_extendibles_and_vp} and a consequence of it. Since a proper class of $C^{(n)}$-extendible cardinals always implies $\VP(\mathbf{\Pi_{n+1}})$ (Theorem~\ref{equivalence_of_vp_and_n_choiceless_extendibles}), failure of Theorem~\ref{bagaria_cn_extendibles_and_vp} can only happen if $\VP(\mathbf{\Pi_{n+1}})$ holds while there are no $C^{(n)}$-extendible cardinals beyond some ordinal $\xi.$ We will show that this failure implies the existence of unboundedly many rank-Berkeley cardinals, and then establish the equiconsistency of these two theories.

An interesting consequence of Theorem~\ref{bagaria_cn_extendibles_and_vp} is that $\boldVP$ implies $\OR$ is Mahlo (Definition~\ref{definition_or_is_mahlo}) \cite[Lemma 6.3]{more_on_preservation}. Since this equivalence may fail in the context of $\ZF,$ and since there is no guarantee so far that $n$-choiceless extendible cardinals are inaccessible, one must wonder whether it is possible to have $\boldVP$ while $\OR$ fails to be Mahlo. We will show that this implies the existence of unboundedly many-rank Berkeley cardinals as well, and then prove the equiconsistency of the two theories.

We start by proving some intermediary results. Lemma~\ref{gluing_choiceless_supercompacts} and Lemma~\ref{creating_choiceless_supercompacts} below are generalizations to $n$-choiceless supercompact cardinals of results by Menachem Magidor \cite[Lemmas 1 \& 2, resp.]{magidor}, modulo slightly stronger assumptions.

\begin{lemma}\label{gluing_choiceless_supercompacts}
    For $n\geq0,$ given $\alpha<\kappa<\delta<\lambda$ such that $(\kappa,\delta)\cap C^{(n)}\neq \emptyset,$ if $\kappa$ is $(\alpha,<\!\delta,<\!V_\delta,n)$-choiceless supercompact and $\delta$ is $(\alpha,\lambda,<\!V_\lambda,n)$-choiceless supercompact, then $\kappa$ is $(\alpha,\lambda,<\!V_\lambda,n)$-choiceless supercompact as well.
\end{lemma}

\begin{proof}
    Fix $a\in V_\lambda.$ We need to show that $\kappa$ is $(\alpha,\lambda,a,n)$-choiceless supercompact. Since $\delta$ is $(\alpha,\lambda,<\!V_\lambda,n)$-choiceless supercompact, we get an elementary embedding $j:V_{\bar\lambda}\rightarrow V_\lambda$ such that $\crit j>\alpha,$\ $\bar\lambda<\delta$ and is in $C^{(n)},$ and there is some $\bar a\in V_{\bar\lambda}$ for which $j(\bar a)=a.$ Now, if it so happens that $\bar \lambda<\kappa,$ then this $j$ witnesses that $\kappa$ is indeed $(\alpha,\lambda,a,n)$-choiceless supercompact.

    Else, if $\bar\lambda = \kappa,$ we proceed as follows. Fix a $\gamma\in (\kappa,\delta)\cap C^{(n)}.$ By using $(\alpha,<\!\delta,<\!V_\delta,n)$-choiceless supercompactness of $\kappa,$ fix an elementary embedding $k:V_{\bar\gamma}\rightarrow V_\gamma$ such that $\crit k>\alpha,$\ $\bar\gamma<\kappa$ and is in $C^{(n)},$ and there are $\dbar a, \dbar\lambda\in V_{\bar\gamma}$ for which $k(\dbar a)=\bar a$ and $k(\dbar\lambda)=\bar\lambda.$ Since $\bar a\in V_{\bar\lambda},$ we must have $\dbar a\in V_{\dbar \lambda}$ by elementarity. Now, the composite map $j\circ k\vert_{V_{\dbar\lambda}}:V_{\dbar\lambda}\rightarrow V_\lambda$ is an elementary embedding with critical point strictly above $\alpha$ and $(j\circ k\vert_{V_{\dbar\lambda}})(\dbar a) = j(\bar a) = a.$ It remains to show that $\dbar\lambda<\kappa$ and is in $C^{(n)}.$ But this is clear, since $\dbar\lambda<\bar \gamma<\kappa,$ while $\dbar\lambda\in C^{(n)}$ by elementarity of $k$ and the fact that $\bar \gamma, \gamma, \bar \lambda \in C^{(n)}.$

    Finally, we need to consider the case where $\bar\lambda>\kappa.$ This is similar to the above, but simpler, because now we do not need a $\gamma,$ and we can just use $\bar\lambda$ itself in place of $\gamma$ in the above argument.
\end{proof}


\begin{definition}
    Given a nontrivial elementary embedding $j:V_\mu\rightarrow V_\nu$ such that $\sup\{\beta\mid j(\beta)=\beta\}<\mu,$ we define the \emph{last sequence} of $j$ to be the longest sequence $\langle\gamma_i\rangle_{i\in m(j)}$ satisfying $\gamma_0=\sup\{\beta\mid j(\beta)=\beta\}$ and $\gamma_i=j(\gamma_{i-1})$ for all nonzero $i\in m(j),$ where $m(j)\leq \omega.$ The ordinal $\gamma_0$ will be called the \emph{last point} of $j.$
\end{definition}


The following lemma is easy.

\begin{lemma}\label{correctness_trick}
    For any $n\geq0,$ if for some ordinals $\alpha<\beta<\gamma$ we have $\gamma \in C^{(n)},$ $\beta\in C^{(n+1)},$ and $\alpha\in (C^{(n+1)})^{V_\gamma},$ then $\alpha\in C^{(n+1)}.$
\end{lemma}

\begin{lemma}\label{creating_choiceless_supercompacts}
    For $n\geq 0,$ if $j:V_\mu\rightarrow V_\nu,$ for some $\mu\in \lim{C^{(n+1)}}$ and $\nu\in C^{(n)},$ is an elementary embedding with critical point $\kappa$ and last sequence $\langle \gamma_i\rangle_{i\in m(j)},$ then $\gamma_0,$ the last point of $j,$ is $(<\!\kappa, <\!\mu, <\!V_\mu, n+1)$-choiceless supercompact.
\end{lemma}

\begin{proof}
    First, let us consider the segment $(\gamma_0,\gamma_1\cap\mu).$ Fix some $\alpha<\kappa,$ $\lambda\in(\gamma_0,\gamma_1\cap\mu)\cap C^{(n+1)},$ and $a\in V_\lambda.$ We must show that $\gamma_0$ is $(\alpha,\lambda,a,n+1)$-choiceless supercompact. Notice that $\mu,\lambda\in C^{(n+1)}$ implies $V_\mu\satisfies ``\lambda\in C^{(n+1)}",$ and so $V_{\nu}\satisfies ``j(\lambda)\in C^{(n+1)}"$ by elementarity. Also, $\nu \in C^{(n)}$ and $\lambda\in C^{(n+1)}$ imply $V_\nu\satisfies ``\lambda\in C^{(n+1)}".$ Therefore, the restricted map $j\vert_{V_\lambda}:V_\lambda \rightarrow V_{j(\lambda)}$ witnesses that $\gamma_1=j(\gamma_0)$ is $(\alpha,j(\lambda),j(a),n+1)$-choiceless supercompact in $V_\nu.$\footnote{$j\vert_{V_\lambda}$ is in $V_\nu,$ as $\nu$ is a limit by elementarity and the fact that $\mu\in \lim{C^{(n+1)}}.$} Again by elementarity, $\gamma_0$ is $(\alpha,\lambda,a,n+1)$-choiceless supercompact in $V_\mu.$ But $\mu$ is correct enough so that any such witness is a witness in $V$ as well. Thus we have shown that $\gamma_0$ is $(<\!\kappa,<\!\gamma_1\cap\mu,<\!V_{\gamma_1\cap\mu},n+1)$-choiceless supercompact.

    If $m(j)=2,$ then $\gamma_1\cap\mu=\mu,$ and so the above argument finishes the proof. However, if $m(j)>2,$ then we just get that $\gamma_0$ is $(<\!\kappa,<\!\gamma_1,<\!V_{\gamma_1},n+1)$-choiceless supercompact. This will serve as the base case for an inductive argument. For the inductive step, we will prove that, for $\gamma_0,$ being $(<\!\kappa,<\!\gamma_m,<\!V_{\gamma_m},n+1)$-choiceless supercompact implies being $(<\!\kappa,<\!\gamma_{m+1},<\!V_{\gamma_{m+1}},n+1)$-choiceless supercompact, whenever $0<m<m(j)-1.$ Since $\sup\{\gamma_i\mid i\in m(j)\}\cap\mu=\mu$ by definition, we will have shown that $\gamma_0$ is $(<\!\kappa, <\!\mu, <\!V_\mu, n+1)$-choiceless supercompact in the case $m(j)>2$ as well.

    So, assume $\gamma_0$ is $(<\!\kappa,<\!\gamma_m,<\!V_{\gamma_m},n+1)$-choiceless supercompact, where $0<m<m(j)-1.$ This last inequality means that $\gamma_m\in V_\mu,$ so, by correctness of $\mu,$\ $V_\mu$ satisfies that $\gamma_0$ is $(<\!\kappa,<\!\gamma_m,<\!V_{\gamma_m},n+1)$-choiceless supercompact as well. By elementarity of $j,$\ $V_\nu$ satisfies that $\gamma_1$ is $(<\!j(\kappa),<\!\gamma_{m+1},<\!V_{\gamma_{m+1}},n+1)$-choiceless supercompact. Since $\kappa<j(\kappa),$ we actually have that
    \begin{equation}\label{in_V_nu_choiceless_case}
        V_\nu\satisfies ``\gamma_1 \textrm{ is $(<\!\kappa,<\!\gamma_{m+1},<\!V_{\gamma_{m+1}},n+1)$-choiceless supercompact}".
    \end{equation}
    
    Let us show that $V$ also satisfies the formula in \eqref{in_V_nu_choiceless_case}. Fix $\alpha<\kappa,\lambda\in C^{(n+1)}\cap (\gamma_1.\gamma_{m+1}),$ and $a\in V_\lambda.$ As $\lambda\in (C^{(n+1)})^{V_\nu},$ by \eqref{in_V_nu_choiceless_case}, we are given an elementary embedding $k:V_{\bar\lambda}\rightarrow V_\lambda$ such that $\crit k>\alpha,$\ $\bar\lambda<\gamma_1$ is in $(C^{(n+1)})^{V_\nu},$ and $j(\bar a)=a$ for some $\bar a \in V_{\bar\lambda}.$ An application of Lemma~\ref{correctness_trick} with $\bar\lambda<\lambda<\nu$ in place of $\alpha<\beta<\gamma$ yields $\bar\lambda\in C^{(n+1)}.$ Hence, $k$ witnesses that $\gamma_1$ is $(\alpha,\lambda,a,n+1)$\-/choiceless supercompact in $V.$

    We have, by inductive assumption, that $\gamma_0$ is $(<\!\kappa,<\!\gamma_m,<\!V_{\gamma_m},n+1)$-choiceless supercompact, and we just showed that $\gamma_1$ is $(<\!\kappa,<\!\gamma_{m+1},<\!V_{\gamma_{m+1}},n+1)$-choiceless supercompact. The proof of the inductive step will be over by an application of Lemma~\ref{gluing_choiceless_supercompacts}, if we can show that $(\gamma_0,\gamma_1)\cap C^{(n+1)}\neq 0.$ We do this by showing $\gamma_1\in\lim C^{(n+1)}.$ First notice that, since $\sup\{\gamma_i\mid i\in m(j)\}\geq\mu$ and $\mu\in\lim C^{(n+1)},$ there exist $\gamma_s>\gamma_1$ and $\delta\in C^{(n+1)}$ such that $\gamma_1<\delta<\gamma_s.$ Fix any $\beta_0\in\{\beta\mid j(\beta)=\beta\}.$ In $V,$ and hence also in $V_\nu,$ $(\beta_0,\gamma_s)\cap C^{(n+1)}\neq \emptyset$ (as witnessed by $\delta).$ By elementarity of $j$ and correctness of $\mu,$ we get $(\beta_0,\gamma_{s-1})\cap C^{(n+1)}\neq \emptyset$ in $V.$ Repeating this argument finitely many times, we get $(\beta_0,\gamma_{0})\cap C^{(n+1)}\neq \emptyset.$ As $\beta_0$ was arbitrary, we have shown that $\gamma_0\in\lim C^{(n+1)}.$ By correctness of $\mu$ and elementarity of $j,$ we have that $\gamma_1\in\lim (C^{(n+1)})^{V_\nu}.$ Finally, since $\gamma_1<\mu<\nu,$\ $\mu\in C^{(n+1)},$ and $\nu\in C^{(n)},$ we may apply Lemma~\ref{correctness_trick} to show that in fact $\gamma_1\in\lim C^{(n+1)}.$
\end{proof}

We can use the above lemmas to prove the following properties of the least $\alpha$-$n$-choiceless extendible cardinal, for any fixed $\alpha.$

\begin{proposition}\label{properties_of_least_alpha_n_extendible}
    For any $n\geq 1,$ if $\alpha$ is an ordinal and $\gamma>\alpha$ is the least $\alpha$-n-choiceless extendible ordinal, then:
    \begin{enumerate}
        \item There exists $\mu_0$ such that for all $\mu\geq\mu_0$ in $\lim{C^{(n+1)}}$ and for all elementary embeddings $j:V_\mu\rightarrow V_\nu$ with $\nu\in C^{(n)},$\ $\crit j>\alpha$ and $j(\gamma)>\mu,$ we must have $\sup\{\beta\mid j(\beta)=\beta\}=\gamma.$
        \item $\gamma\in \lim C^{(n+1)}.$
        \item There is no cofinal map $f:V_\xi\rightarrow \gamma,$ for any $\xi\leq\alpha.$ In particular, $\cof \gamma > \alpha.$
    \end{enumerate}
\end{proposition}

\begin{proof}
    Part (i): Suppose that this is not the case. By $\alpha$-$n$-choiceless extendibility of $\gamma,$ we have a proper class of $\mu\in \lim{C^{(n+1)}}$ for which there are elementary embeddings $j:V_\mu\rightarrow V_\nu$ with $\nu\in C^{(n)},$\ $\crit j>\alpha,$\ $j(\gamma)>\mu,$ and $\sup\{\beta\mid j(\beta)=\beta\}<\gamma.$ This means that there exist $\kappa$ and $\delta,$ satisfying $\alpha<\kappa\leq\delta<\gamma,$ and there is a proper class of $\mu\in \lim{C^{(n+1)}}$ for which there are elementary embeddings $j:V_\mu\rightarrow V_\nu$ with $\nu\in C^{(n)},$\ $\crit j=\kappa,$\ $j(\gamma)>\mu,$ and $\sup\{\beta\mid j(\beta)=\beta\}=\delta.$ Thus, for any such $j$ the ordinal $\delta$ is the last point. Hence, by Lemma~\ref{creating_choiceless_supercompacts}, we must have that $\delta$ is $\alpha$-$n+1$-choiceless supercompact. By Lemma~\ref{choiceless_supercompact_to_choiceless_extendible}, we then have that $\delta<\gamma$ is $\alpha$-$n$-choiceless extendible. But this cannot be by minimality of $\gamma.$

    Part (ii): Fix an elementary embedding $j:V_\mu\rightarrow V_\nu$ as in (i). Let $\xi<\gamma$ be arbitrary. Fix $\beta$ such that $\xi<\beta<\gamma$ and $j(\beta)=\beta.$ Notice that $V_\nu$ satisfies $(\beta,j(\gamma))\cap C^{(n+1)}\neq \emptyset$ (as witnessed by $\mu).$ By elementarity of $j$ and correctness of $\mu,$ we have that $(\beta,\gamma)\cap C^{(n+1)}\neq \emptyset.$

    Part (iii): Assume towards a contradiction that $f:V_\xi\rightarrow \gamma$ is cofinal for some $\xi\leq\alpha.$ Fix $j$ as in (i). Define $f^*:V_\xi\rightarrow \gamma$ so that $f^*(x)$ is the least $\beta$ above $f(x)$ such that $j(\beta)=\beta.$ Then, $f^*$ must also be a cofinal map in $\gamma.$ Hence, by elementarity of $j,$ the map $j(f^*):V_\xi\rightarrow j(\gamma)$ must be a cofinal map in $j(\gamma).$ But that cannot be, since $j(\gamma)>\gamma,$ while $j(f^*)(x)=j(f^*)(j(x))= j(f^*(x)) = f^*(x),$ for all $x\in V_\xi.$
\end{proof}

\begin{definition}[\cite{cn-cardinals}]\label{cn_extendibles}
    For each $n\geq0,$ a cardinal $\kappa$ is said to be \emph{$\mu$-$C^{(n)}$\-/extendible} for some ordinal $\mu>\kappa$ in $C^{(n)}$ iff there is a $\nu>\mu$ in $C^{(n)}$ and an elementary embedding $j:V_\mu \rightarrow V_\nu$ such that $\crit j=\kappa$ and $j(\kappa)>\mu.$
    $\kappa$ is said to be \emph{${<}\delta$-$C^{(n)}$\-/extendible} iff it is $\mu$-$C^{(n)}$\-/extendible for all $\kappa<\mu<\delta$ in $C^{(n)}.$ Finally, $\kappa$ is said to be \emph{$C^{(n)}$\-/extendible} iff it is $\mu$-$C^{(n)}$\-/extendible for all $\mu>\kappa$ in $C^{(n)}.$\footnote{It is easy to see that this notion is equivalent to Bagaria's $C^{(n)}$-extendibility \cite{cn-cardinals} by an argument similar to the one leading to Theorem~\ref{equivalence_of_choiceless_extendibles_and_choiceless_supercompacts}.}
\end{definition}

\begin{theorem}\label{vp_no_extendibles}
    For $n\geq1,$ if $\VP(\mathbf{\Pi_{n+1}})$ holds while there are no $C^{(n)}$-extendible cardinals above some ordinal $\xi,$ then there are unboundedly many rank-Berkeley cardinals.
\end{theorem}

\begin{proof}
    Let $\alpha>\xi$ be an arbitrary ordinal, and let $\gamma>\alpha$ be the least $\alpha$-$n$-choiceless extendible ordinal. By Proposition~\ref{properties_of_least_alpha_n_extendible}, part (i), there are elementary embeddings $j:V_\mu \rightarrow V_\nu$ satisfying $\crit j >\alpha,$\ $j(\gamma)>\mu,$ and $\sup\{\beta\mid j(\beta)=\beta\}=\gamma,$ for arbitrarily high $\mu,\nu\in C^{(n)}.$ If all such $j$ have critical points equal to $\gamma,$ then clearly $\gamma$ would be a $C^{(n)}$-extendible cardinal, contrary to the fact that there are no $C^{(n)}$-extendible cardinals above $\xi.$ Hence, there must be some elementary embedding $j:V_\mu\rightarrow V_\nu$ satisfying $\alpha<\crit j <\gamma,$\ $j(\gamma)>\mu,$ and $\sup\{\beta\mid j(\beta)=\beta\}=\gamma,$ for some $\mu,\nu\in C^{(n)}.$ Define $\lambda<\gamma$ as the first fixed point of $j$ above $\crit j.$

    By Proposition~\ref{properties_of_least_alpha_n_extendible}, part (iii), cofinality of $\gamma$ is greater than $\omega.$ Hence, the set $\{\beta \mid j(\beta)=\beta\}$ forms an $\omega$-club below $\gamma.$ Also, by part (ii) of the same proposition, $C^{(n+1)}\cap\gamma$ must form an $\omega$-club below $\gamma$ too. Therefore, we see that $C^{(n+1)}\cap\{\beta \mid j(\beta)=\beta\}$ is nonempty and, in fact, unbounded in $\gamma.$ Let $\delta>\lambda$ be an ordinal in this intersection and notice that we now have an elementary embedding $j\vert_{V_\delta}:V_\delta\rightarrow V_\delta.$

    We claim that $V_\delta\satisfies ``\lambda\ \text{is rank-Berkeley}".$ Otherwise, there would be a least counterexample $\sigma\in V_\delta.$ $\sigma$ is definable from $\lambda$ in $V_\delta,$ and since $j\vert_{V_\delta}$ fixes $\lambda,$ it must also fix $\sigma.$ But, the restriction of $j\vert_{V_\delta}$ to $V_\sigma$ will give a contradiction, so the claim is correct. Now, as $\delta\in C^{(n+1)}\subset C^{(2)}$ and being rank-Berkeley is a $\Pi_2$ statement, the cardinal $\lambda$ must be rank-Berkeley in $V$ too. As $\alpha$ was chosen arbitrarily and $\lambda>\alpha,$ we get that there are arbitrarily high rank-Berkeley cardinals.
\end{proof}

\begin{corollary}\label{equiconsistency_vp_no_extendibles_with_many_rank_berkeleys}
    For $n\geq1,$ the following theories are equiconsistent:
    \begin{enumerate}
        \item $\ZF + \boldVP + ``\forall \kappa (\kappa \textrm{\textup{ is not $C^{(0)}$\-/extendible}})"$
        \item $\ZF + \VP(\mathbf{\Pi_{n+1}}) + ``\forall \kappa (\kappa \textrm{\textup{ is not $C^{(n)}$\-/extendible}})"$
        \item $\ZF + \VP(\mathbf{\Pi_{n+1}}) + ``\exists \xi \forall \kappa>\xi (\kappa \textrm{\textup{ is not $C^{(n)}$\-/extendible}})"$
        \item $\ZF + ``\textrm{\textup{There are unboundedly many rank-Berkeley cardinals}}"$
    \end{enumerate}
\end{corollary}

\begin{proof}
    (i) to (ii) and (ii) to (iii) are trivial. (iii) to (iv) is by Theorem~\ref{vp_no_extendibles}.

    From (iv) to (i): Work in (iv). We can assume that there is no inaccessible cardinal $\kappa$ that is a limit of rank-Berkeley cardinals by simply passing to $V_\kappa,$ where $\kappa$ is the least such cardinal if it exists. It is easy to see that any $C^{(0)}$-extendible cardinal is both inaccessible and a limit of rank-Berkeley cardinals, and so cannot exist. Meanwhile, $\boldVP$ holds, by Proposition~\ref{rank_berkeley_to_vp} and Corollary~\ref{vp_to_boldvp}.
\end{proof}

\begin{definition}\label{definition_or_is_mahlo}
    We define the axiom schema $``\OR \textrm{\textup{ is Mahlo}}"$ to mean the following: Every definable, with parameters, proper club class of ordinals has an inaccessible cardinal. The negation of this, $``\OR \textrm{\textup{ is not Mahlo}},"$ will be a single axiom stating that some definable, with parameters, proper club class of ordinals contains no inaccessible cardinal.
\end{definition}

\begin{theorem}\label{or_is_not_mahlo_proves_unbounded_rank_berkeleys}
If $\boldVP$ holds but $\OR$ is not Mahlo, then there are unboundedly many rank-Berkeley cardinals.
\end{theorem}

\begin{proof}
    Fix $C$ and $n$ such that $C$ is a $\mathbf{\Sigma_n}$ club class of non-inaccessible ordinals. Let $m\geq \max\{n,1\}$ and, for some arbitrary ordinal $\alpha,$ let $\gamma$ be the least $\alpha$-$m$-choiceless extendible ordinal. Using Proposition~\ref{properties_of_least_alpha_n_extendible} and $\alpha$-$m$-choiceless extendibility of $\gamma,$ fix an elementary $j:V_\mu \rightarrow V_\nu$ with $\alpha<\crit{j},$ $j(\gamma)>\mu,$ and $\sup\{\beta \mid j(\beta)=\beta\}=\gamma.$
    
    Since $\gamma\in\lim C^{(m+1)}$ and $m\geq n,$ we get that $C\cap \gamma$ is unbounded in $\gamma,$ and so $\gamma\in C.$ As members of $C$ are not inaccessible, we have $\crit{j}< \gamma.$ Let $\lambda=\min\{\beta>\crit j\mid j(\beta)=\beta\}.$ Now, we continue as in paragraphs 2-3 of the proof of Theorem~\ref{vp_no_extendibles}, with $m$ replacing $n.$
\end{proof}

\begin{corollary}\label{equiconsistency_theories_or_is_not_mahlo}
    The following theories are equiconsistent:
    \begin{enumerate}
        \item $\ZF + \boldVP + ``\OR \textrm{\textup{ is not Mahlo}}"$
        \item $\ZF + ``\textrm{\textup{There are unboundedly many rank-Berkeley cardinals}}"$
    \end{enumerate}
\end{corollary}

\begin{proof}
    (i) to (ii): Theorem~\ref{or_is_not_mahlo_proves_unbounded_rank_berkeleys}. (ii) to (i): Suppose there are unboundedly many rank-Berkeley cardinals. Just as in the proof of case (ii) to (iii) of Theorem~\ref{equiconsistency_vp_no_extendibles_with_many_rank_berkeleys}, $\boldVP$ holds and we can assume that there is no inaccessible cardinal $\kappa$ that is a limit of rank-Berkeley cardinals. Now, the club class $C$ consisting of limits of rank-Berkeley cardinals contains no inaccessible cardinals, hence $\OR$ is not Mahlo.
\end{proof}

\printbibliography

@article{lcbc,
    author={Joan Bagaria AND Peter Koellner AND W. Hugh Woodin},
    title={Large cardinals beyond choice},
    journal={The Bulletin of Symbolic Logic},
    volume={25},
    number={3},
    year={2019},
    pages={283-318},
    publisher={Cambridge University Press}
}

@phdthesis{reinhardt,
    author = {William Reinhardt},
    title = {Topics in the Metamathematics of Set Theory},
    school = {University of California, Berkeley},
    year = {1967}
}

@article{kunen,
    author = {Kenneth Kunen},
    journal = {The Journal of Symbolic Logic},
    number = {3},
    pages = {407-413},
    publisher = {Association for Symbolic Logic},
    title = {Elementary embeddings and infinitary combinatorics},
    volume = {36},
    year = {1971}
}

@article{cn-cardinals,
	author = {Joan Bagaria},
	journal = {Archive for Mathematical Logic},
	number = {3-4},
	pages = {213-240},
	publisher = {Springer Verlag},
	title = {$C^{(n)}$-Cardinals},
	volume = {51},
	year = {2012}
}

@article{more_on_preservation,
    title = {More on the preservation of large cardinals under class forcing},
    volume = {88},
    number = {1},
    journal = {The Journal of Symbolic Logic},
    author = {Joan Bagaria AND Alejandro Poveda},
    year = {2023},
    pages = {290–323}
}

@article{magidor,
    title = {On the role of supercompact and extendible cardinals in logic},
    volume = {10},
    issue = {2},
    journal = {Israel Journal of Mathematics},
    author = {Menachem Magidor},
    year = {1971},
    pages = {147–157}
}

@misc{aspero,
    title = {A short note on very large large cardinals (without choice)},
    author = {David Asper\'o},
    URL = {https://archive.uea.ac.uk/~bfe12ncu/asnovslcwc.pdf}
}

@article{goldberg_measurable_choiceless,
    title = {Measurable cardinals and choiceless axioms},
    journal = {Annals of Pure and Applied Logic},
    volume = {175},
    number = {1, Part B},
    pages = {103323},
    year = {2024},
    author = {Gabriel Goldberg},
}

@misc{schlutzenberg,
      title={A weak reflection of Reinhardt by super Reinhardt cardinals}, 
      author={Farmer Schlutzenberg},
      year={2020},
      eprint={2005.11111},
      archivePrefix={arXiv},
      primaryClass={math.LO}
}

@book{kanamori,
    author = {Akihiro Kanamori},
    title = {The Higher Infinite},
    publisher = {Springer Berlin, Heidelberg},
    year = {2003}
}

@InBook{dehornoy,
    title     = {Elementary Embeddings and Algebra},
    publisher = {Springer Netherlands},
    year      = {2009},
    author    = {Patrick Dehornoy},
    editor    = {Matthew Foreman and Akihiro Kanamori},
    type      = {Book Chapter},
    booktitle = {Handbook of Set Theory}
}

\end{document}